\documentclass[11pt]{article}

\usepackage[utf8]{inputenc}
\usepackage[vietnamese,english]{babel}
\usepackage[all, cmtip]{xy}

\usepackage{amsmath}
\usepackage{amsxtra}
\usepackage{amscd}
\usepackage{amsthm}
\usepackage{amsfonts}
\usepackage{amssymb}
\usepackage{bbm}

\usepackage[bitstream-charter]{mathdesign}
\usepackage[T1]{fontenc}
\usepackage{fancyhdr}
\usepackage{tikz}
\usepackage{tikz-cd}
\usetikzlibrary{matrix,calc,arrows}

\numberwithin{equation}{section}

\newtheorem{theorem}{Theorem}
\newtheorem{lemma}[theorem]{Lemma}
\newtheorem{proposition}[theorem]{Proposition}
\newtheorem{definition}[theorem]{Definition}
\newtheorem{corollary}[theorem]{Corollary}

\numberwithin{theorem}{section}
\theoremstyle{definition}

\theoremstyle{remark}
\newtheorem{remark}{Remark}[section]
\newtheorem{example}[remark]{Example}

\newcommand\indlim\varinjlim

\newcommand\cA{\mathcal{A}}
\newcommand\cB{\mathcal{B}}

\newcommand\cI{\mathcal{I}}

\newcommand\cM{\mathcal{M}}

\newcommand\cO{\mathcal{O}}
\newcommand\cP{\mathcal{P}}

\newcommand\cT{\mathcal{T}}

\let\fg\undefined

\newcommand{\fg}{{\mathfrak g}}

\newcommand{\fm}{{\mathfrak m}}

\newcommand{\ft}{{\mathfrak t}}

\newcommand{\fC}{{\mathfrak C}}

\newcommand{\fS}{{\mathfrak S}}

\newcommand\bbA{\mathbb{A}}
\newcommand\BB{\mathbb{B}}

\newcommand\DD{\mathbb{D}}

\newcommand\bA{{\bbA}}

\newcommand\Z{\mathbb{Z}}

\newcommand\N{\mathbb{N}}

\def\d{\mathrm{d}}

\newcommand\rH{\mathrm{H}}

\newcommand\rT{\mathrm{T}}

\newcommand\GL{\mathrm{GL}}

\newcommand\tr{\mathrm{tr}}

\newcommand\Spec{\mathrm{Spec}}

\newcommand\Gm{\mathbb{G}_m}

\newcommand\End{\mathrm{End}}

\newcommand\Pic{\mathrm{Pic}}
\newcommand\Sym{\mathrm{S}}

\newcommand\ad{\mathrm{ad}}

\newcommand\Supp{\mathrm{Supp}}

\newcommand\Chow{\mathrm{Chow}}
\newcommand\Hilb{\mathrm{Hilb}}
\newcommand\HC{\mathrm{HC}}

\newcommand\length{\mathrm{length}}

\newcommand{\mM}{\mathcal{M}}

\newcommand{\mO}{\mathcal{O}}

\newcommand{\sB}{\mathscr{B}}
\newcommand{\sP}{\mathscr{P}}

\newcommand{\sA}{\mathscr{A}}

\newcommand{\on}{\operatorname}

\newcommand{\tX}{\widetilde X}

\newcommand{\is}{\simeq}

\newcommand{\tC}{\widetilde C}

\newcommand{\quash}[1]{}  
\newcommand{\nc}{\newcommand}

\newcommand{\bbB}{{\mathbb B}}

\newcommand{\Higgs}{{\on{Higgs}}}
\newcommand\sd{\on{sd}}

\nc{\al}{{\alpha}} \nc{\be}{{\beta}} \nc{\ga}{{\gamma}}
\nc{\ve}{{\varepsilon}} 
\nc{\La}{{\Lambda}}

\newcommand{\beqn}{\begin{equation*}}
\newcommand{\eeqn}{\end{equation*}}

\newcommand{\beq}{\begin{equation}}
\newcommand{\eeq}{\end{equation}}

\title{On the Hitchin fibration for algebraic surfaces}
\author{T.H. Chen, B.C. Ngô}

\date{}

\begin{document}

\maketitle

\begin{abstract}
In this paper we explore the structure of the Hitchin fibration for higher dimensional varieties with emphasis on the case of algebraic surfaces.
\end{abstract}

\section{Introduction}
In the celebrated work \cite{Hitchin}, Hitchin constructs a completely integrable system on the moduli space of Higgs bundles over a compact Riemann surface $X$. This system can be presented as a morphism $h_X:\cM_X\to\cA_X$ where $\cM_X$ is the moduli space of Higgs bundles and $\cA_X$ is a certain affine space of half the dimension of $\cM_X$. For every $a\in\cA_X$, Hitchin constructs a spectral curve $Y_a$ which is a finite flat covering of $X$, embedded in the cotangent bundle of $X$ and gives a description of the fiber $h_X^{-1}(a)$ in terms of the Picard variety of $Y_a$. As a result, he proves that connected components of the generic fiber of the Hitchin map $h_X$ are abelian varieties.

In this paper, we explore the structure of the Hitchin map for higher dimensional varieties with emphasis on the case of surfaces. 
Let $k$ be an algebraically closed field of characteristic zero.
In \cite{Simpson 1}, Simpson defines a Higgs bundle over a smooth projective variety $X$ over $k$ as a pair $(E,\theta)$ where $E$ is a vector bundle of degree $n$ equipped with a Higgs field
$\theta:E\to E\otimes_{\cO_X} \Omega^1_{X}$
which is a $\cO_X$-linear map satisfying the integrability equation
\begin{equation}\label{integrability}
	\theta\wedge\theta=0.
\end{equation}
In order to unravel the integrability equation, we recall that $\theta\wedge \theta$ is defined as the composition 
of $\theta:E\to E\otimes_{\cO_X} \Omega^1_{X}$ with the
map $E\otimes_{\cO_X} \Omega^1_{X/k} \to E\otimes_{\cO_X} \Omega^2_{X}$ given by $v\otimes \omega \mapsto \theta(v)\wedge \omega$. After choosing  local coordinates $z_1,\ldots,z_d$ of $X$ around $x\in X$, we can write $\theta$ uniquely as a sum $\theta(v)=\sum_{i=1}^d \theta_i(v) \otimes \d z_i$
where $\theta_i(v)$ are local endomorphisms of $E$. Now, a direct calculation shows 
\begin{equation*}
(\theta\wedge \theta)(v)=\sum_{i<j} [\theta_j,\theta_i](v) \d z_i \wedge  \d z_j
\end{equation*}
and in particular the integrability equation holds if and only if the matrices $\theta_i$ commute with each other.

If $G$ is a reductive group a $G$-bundle with Higgs field on $X$ is a pair $(E,\theta)$ where $E$ is a $G$-bundle over $X$ equipped with a Higgs field
\begin{equation} \label{G-Higgs-field}
\theta\in \rH^0(X,\ad(E) \otimes \Omega^1_{X})
\end{equation}
where $\ad(E)$ is the adjoint vector bundle of $E$. The Higgs field $\theta$ is to satisfy the integrability equation \eqref{integrability}. This means if we have local coordinate $z_1,\ldots, z_d$ and a trivialization of $E$ on a neighborhood $U$ of $x\in X$, then on $U$ we can write $\theta=\sum_{i=1}^d \theta_i \d z_i$
where $\theta_i : U\to \fg$ are functions on $U$ with values in the Lie algebra $\fg$ of $G$ satisfying the integrability equation
$[\theta_i,\theta_j]=0$ for all $1\leq i,j \leq d$. 

We denote $\fC_G^d$ the closed subscheme of $\fg^d$ consisting of $(x_1,\ldots,x_d)\in\fg^d$ such that $[x_i,x_j]=0$ for all indices $i,j$. We consider the diagonal adjoint action of $G$ on $\fC_G^d$ given by 
\begin{equation}\label{diagonal-adjoint}
(x_1,\ldots,x_d) \mapsto (\ad(g)x_1,\ldots,\ad(g)x_d)
\end{equation}
and the commuting action of $\GL_d$ given by
\begin{equation}\label{GLd-acts-on-vectors}
	(x_1,\ldots,x_d) \mapsto (x_1,\ldots,x_d) h
\end{equation}
for every $h\in \GL_d$. We will consider the quotient stack
\begin{equation}
	[\fC_G^d/(G\times \GL_d)].
\end{equation}
For a test scheme $S$, a $S$-point of $[\fC_G^d/(G\times \GL_d)]$ consists in a vector bundle $V$ of rank $d$ over $S$, a $G$-bundle $E$ over $S$, a $\cO_S$-linear morphism $\theta:V\to \ad(E)$ from $V$ to the adjoint vector bundle of $E$ such that for all local sections $v_1,v_2$ of $V$, we have $[\theta(v_1),\theta(v_2)]=0$.   
A Higgs $G$-bundle on a smooth $d$-dimensional algebraic variety $X$ can be then interpreted as a map
\begin{equation}\label{Higgs-surface}
\theta: X \to [\fC_G^d/(G\times \GL_d)]
\end{equation}
such that the induced map $X\to \BB\GL_d$ from $X$ to the classifying space of $\GL_d$ represents the cotangent bundle of $X$.

Following the strategy of \cite{Ngo}, we will study the Hitchin map through the morphism 
\beq\label{Chevalley map}
[\fC^d_G/G]\to\fC_G^d/G
\eeq
from the stack $[\fC_G^d/G]$ to the coarse quotient $\fC_G^d/G$. By definition, the coarse quotient $\fC_G^d/G=\Spec(k[\fC_G^d]^G)$ is the spectrum of of the ring of $G$-invariant functions on $\fC_G^d$. In the case $d=1$, the Chevalley restriction theorem asserts an isomorphism $k[\fg]^G=k[\ft]^W$ where $\ft$ is a Cartan algebra of $\fg$ and $W$ the Weyl group. We also know that $k[\ft]^W$ is a polynomial algebra. 
Thus we have a $\GL_1$-equivariant morphism 
\[[\fg/G]\to\Spec(k[\ft]^W)\]
and, for a smooth curve $X$, the Hitchin map attaches to $\theta$ in (\ref{Higgs-surface}) the induced morphism 
$a:X\to[\Spec(k[\ft]^W)/\GL_1]$, which is the characteristic polynomial of $\theta$
in an appropriate sense.

Thanks to work of Gan-Ginzburg and Joseph \cite{GG,Joseph}, the analogue of the Chevalley restriction theorem holds in the case $d=2$ and $G=\GL_2$:
\begin{equation} \label{Joseph}
	k[\fC_G^2]^G = k[\ft\times\ft]^W.
\end{equation}
In general, $k[\ft\times\ft]^W$ is not a polynomial algebra, and not even a smooth $k$-algebra, although it is always a Cohen-Macaulay ring after a theorem of Hochshter-Robert \cite{HR}.
For an arbitrary number $d$, by restricting functions on $\fC_G^d$ to $\ft^d$, we obtain a 
$\GL_d$-equivariant homomorphism 
\begin{equation} \label{Chevalley-d}
	k[\fC_G^d]^G \to k[\ft^d]^W.
\end{equation} 
It is conjectured that (\ref{Chevalley-d}) is an isomorphism.
In \cite{Hunziker}, Hunziker proves that this homomorphism induces an isomorphism 
from the reduced quotient of $k[\fC_G^d]^G$ to $k[\ft^d]^W$
when $G$ is either a classical group or of type $G_2$. 

As we shall see, in this case $G=\GL_n$, there is a canonical $\GL_d$-equivariant lifting  
\begin{equation}\label{spectral}
	[\fC_G^d/G] \to \Spec(k[\ft^d]^W)
\end{equation}
of $[\fC_G^d/G] \to \fC_G^d/G$ in~\eqref{Chevalley map} along the morphism 
$\Spec(k[\ft^d]^W)\to\fC_G^d/G$
induced from~\eqref{Chevalley-d}. 
Thus for every $d$-dimensional smooth variety $X$, by composing with (\ref{spectral}), we derive a morphism 
 \begin{equation}\label{sd-map}
	\sd_X:\Higgs_{X} \to \cA_X
\end{equation}
where $\Higgs_{X}$ is the moduli stack of all Higgs $\GL_n$-bundles as in \eqref{Higgs-surface} and $\cA_X$ is the space of maps $X\to [\Spec(k[\ft^d]^W)/\GL_d]$ laying over the morphism $X\to \BB\GL_d$ corresponding to the cotangent bundle of $X$. 
We will call $\cA_X$ the space of spectral data and 
$\sd_X$ the spectral data morphism.

In section 2, we will  investigate the relation between our spectral data morphism and the Hitchin map defined by Simpson in \cite{Simpson 2}. 
We show that the Hitchin map factors canonically through the spectral data morphism. 
In section 3, we will construct the spectral cover associated with each spectral datum. With help of this spectral cover, we will provide a description of the generic fiber of the Hitchin map in a similar manner to Beauville-Narasimhan-Ramanan's description \cite{BNR} in the one-dimensional case. This description is of limited use as the spectral cover is not Cohen-Macaulay in general. It thus arises the need of constructing a finite Cohen-Macaulayfication of the spectral cover. This is the purpose of section 4 where we restrict ourselves to the two-dimensional case. The upshot is that in the surface case, as we can construct a canonical Cohen-Macaulayfication of the spectral cover, the Hitchin fibers over generic spectral data have a description fairly similar to the one-dimensional case. The description of the space $\cA_X$ of spectral data remains nevertheless quite difficult and seems to depend on the Enriques classification of surfaces as surveyed in \cite{Beauville}. We investigate this matter in the cases of abelian, ruled and elliptic surfaces in the last sections of the paper. 

In this paper, we restrict ourselves to the case $G=\GL_n$.
All the constructions can be carried out for general reductive groups but technical lemmas still need to be proved. We plan to attend to this matter in a future work. 

Throughout this paper, $k$ is  an algebraically closed field of characteristic $p>n$ or of characteristic zero. All schemes are defined over $k$.
 Without explicitly mentioned otherwise, points will mean $k$-points.  

\section{Spectral data for $\GL_n$}

In this section we study the spectral data morphism for $\GL _n$. We begin by investigating spectral data associated to $d$ commuting $n\times n$-matrices. Instead of $d$ commuting $n\times n$-matrices we will consider a $d$-dimensional $k$-vector space $T$, a $n$-dimensional $k$-vector space $E$ equipped with a $k$-linear map $\theta:T\to \End(E)$ such that for all vectors $x_1,x_2\in T$ we have $[\theta(x_1),\theta(x_2)]=0$. A map $\theta:T\to \End(E)$ satisfying this property will be called a commuting map. 

We will define the spectral datum of $(T,E,\theta)$ as follows. The commuting map $\theta:T\to \End(E)$ defines an action of the symmetric algebra $\Sym(T)$ on $E$. In other words, it produces a $\Sym(T)$-module $F$ whose underlying vector space is $E$. Since $E$ a is finite dimensional vector space, $F$ is supported by finitely many points in the affine space $V=\Spec (\Sym(T))$: we have a decomposition $F=\bigoplus_{\alpha\in V} F_\alpha$ where $F_\alpha$ is a $\Sym(T)$-module annihilated by some power of the maximal ideal $\fm_\alpha$ corresponding to the point $\alpha$. 
Let $\Chow_n(V)$ be the Chow scheme classifying $0$-cycle of length $n$ in $V$.
The decomposition above gives rise to a $0$-cycle 
$$Z=\sum_{\alpha\in V} \lg(F_\alpha)\alpha$$
of length $n$ in $V$. 
We call $Z\in\Chow_n(V)$ the spectral datum of $(T,E,\theta)$.


The above discussion can be generalized over an arbitrary scheme $X$. We can construct the spectral datum attached to $(T,E,\theta)$ where $T$ and $E$ are vector bundles over $X$ of rank $d$ and $n$, and $\theta:T\to \End_{\cO_X}(E)$ is a commuting map. 
Let $V=\Spec_X(\Sym_{\cO_X}(T))$ be the dual bundle of $T$. 
The commuting map extends to a morphism $\Sym_{\cO_X}(T)\to\End_{\cO_X}(E)$. 
Thus there is a coherent sheave 
$F$ on $V$ such that $p_*F=E$  where $p:V\to X$ is the projection map.
Let $\on{Chow}_n(V/X)$ be the relative Chow scheme of $0$-cycle of length $n$ in $V$.
Then as explained in \cite[Section 9.3, p 254]{BLR}, one can associate to 
such $F$ a morphism  
\[
a_F:X\mapsto \on{Chow}_n(V/X),
\]
which is compatible with any base change $X'\to X$ and
for every point $x\in X$ the image $a_F(x)$ is the 
spectral datum $Z$ of the triple $(T,E,\theta)|_x$. 
We call $a_F$ the spectral datum of $(T,E,\theta)$.

Note that any triple $(T,E,\theta)$ as above can be interpreted as
a morphism \[X\to [\fC_{\GL_n}^d/(\GL_n\times\GL_d)]\]
laying over the morphism $S\to\bbB\GL_d$ corresponding to the dual of the $d$-dimensional vector bundle $T$, and 
the spectral datum $a_F$ of $(T,E,\theta)$ can be interpreted as a morphism \[X\to [\Chow_n(\mathbb A^d)/\GL_d]\]
laying over the morphism $X\to\bbB\GL_d$ corresponding to the 
dual of the 
$d$-dimensional vector bundle $T$. 
Then the above construction provides us with a morphism
\beq\label{Chevalley-d GL_n}
[\fC_{\GL_n}^d/(\GL_n\times\GL_d)] \to [\Chow_n(\mathbb A^d)/\GL_d]
\eeq
which is the morphism in (\ref{spectral}) (or rather its quotient by $\GL_d$).

\begin{remark}
In the case $d=2$ and $k$ is of characteristic zero, by the works of Gan-Ginzburg and Joseph (see, e.g., \cite[Theorem 1.2.1]{GG}), 
the Chevalley restriction map  
induces an 
isomorphism 
$\fC_{\GL_n}^2/\GL_n\is\Chow_n(\mathbb A^{2})$ and from it we have a morphism 
  \[[\fC_{\GL_n}^2/\GL_n\times\GL_2]\to[\Chow_n(\mathbb A^{2})/\GL_2]\]
which is the morphism (\ref{Chevalley-d GL_n}) in the case $d=2$.
At the moment, we don't know if their work can be generalized to positive characteristic. 
The discussion above provides a construction of (\ref{Chevalley-d GL_n}) which is valid in any 
characteristic.
\end{remark}

We will apply it to the case where $X$ is smooth of dimension $d$ and $T=\cT_X$ is its tangent bundle. In this case we recover the concepts of Higgs bundles and the spectral data morphism.

In more detail. 
Let $(E,\theta)$ be a Higgs bundle where $E$ is a vector bundle of rank $n$ over $X$ and $\theta:E\to E\otimes_{\cO_X} \Omega^1_{X}$ is a Higgs field.
The Higgs field $\theta$ can be regarded as a $\cO_X$-linear map
\begin{equation} \label{tangent-action}
\theta:\cT_{X} \to \End_{\cO_X}(E)
\end{equation}
where $\cT_{X}$ is the tangent sheaf of $X$ dual to the cotangent sheaf $\Omega^1_{X}$. The integrability condition $\theta\wedge\theta=0$ guarantees that \eqref{tangent-action} induces an action of the symmetric algebra of $\cT_{X/k}$ over $\cO_{X/k}$ i.e. a $\cO_{X}$-linear map $\theta:\Sym_{\cO_{X/k}}(\cT_{X/k}) \to\End_{\cO_X}(E)$. In other words, there is a canonical isomorphism $E=p_* F$ where $F$ is a cohorent module over $\rT^*_X=\Spec_X \Sym_{\cO_{X/k}}(\cT_{X/k})$ the total space of the cotangent bundle and $p:\rT^*_X\to X$ is the projection map.

For every point $x\in X$, $E_x$ is a $n$-dimensional $k$-vector spaces equipped with an action of the symmetric algebra $\Sym_k(\cT_{X,x})$ where $\cT_{X,x}$ is the tangent space of $X$ at $x$. We denote $F_x$ the vector space $E_x$ but regarded as $\Sym_k(\cT_{X,x})$-module. As $\Spec (\Sym_k(\cT_{X,x})) = \rT^*_{X,x}$, there is a canonical decomposition
$F_x =\bigoplus_{\alpha\in \rT^*_{X,x}} F_\alpha$
where $F_\alpha$ is a finite length module supported by the point $\alpha\in T_x^*$. The function
\begin{equation} \label{0-cycle-in-cotangent}
x\mapsto \sum_\alpha n_\alpha \alpha
\end{equation}
where $n_\alpha=\lg(F_\alpha)$ is the dimension of $F_\alpha$ as $k$-vector space, defines a map from $X$ with values in the space of $0$-cycles of $\rT^*_{X,x}$. 
We define the {\em spectral datum} of a Higgs bundle $(E,\theta)$ to be the section
\[a:X \to \Chow_n (\rT^*_X/X)\] given by the formula \eqref{0-cycle-in-cotangent}. Here, the relative Chow variety $\Chow_n (\rT^*_X/X)$ constructed from $\Chow_n(\mathbb A^d)$ by twisting it with the $\GL_d$-bundle attached to the cotangent bundle $\rT^*_X$. This provides us with a concrete description of the spectral data morphism \[\sd_X:\Higgs_X\to\cA_X\] 
in \eqref{sd-map} in the $\GL_n$ case.

We shall investigate the relation between our spectral data morphism and the Hitchin map defined by Simpson.
Let $V$ be a finite dimensional vector space.
We recall that the coordinate ring of $\Chow_n(V)$ is the ring of invariants with respect to the action of the symmetric group $\fS_n$ on the symmetric algebra $\Sym(T^n)$ of $T^n$, where $T^n$ is the $n$-fold direct sum $T\oplus\cdots\oplus T$. By a general theorem of Hochster-Roberts \cite{HR}, this ring is Cohen-Macaulay. A point of $\Chow_n(V)$ will be represented by an unordered collection of $n$ points
$$[v_1,\ldots,v_n] \in \Chow_n(V)$$
with $v_1,\ldots,v_n\in V$ not necessarily distinct. We may also represent it as
$$[u_1^{n_1},\ldots,u_m^{n_m}]\in \Chow_n(V)$$
where $u_1,\ldots,u_m$ are distinct vectors of $V$ and $n_i$ is the number of occurrences of $u_i$ in $[v_1,\ldots,v_n]$.  

We will define an affine embedding of $\Chow_n(V)$. 

\begin{lemma}\label{closed-immersion}
For every finite dimensional vector space $V$ the morphism 
\begin{equation}\label{characteristic-polynomial}
	\iota:\Chow_n(V) \to V \times \Sym^2 V \times \cdots \times \Sym^n V
\end{equation}
given by 
\begin{equation} \label{chi-coeff}
[v_1,\ldots,v_n]\mapsto (a_1,\ldots,a_n)
\end{equation}
where $a_1,\ldots,a_n$ are the symmetric tensors 
\begin{eqnarray*}
	a_1 & = & v_1+\cdots+v_n, \\
	a_2 & = & v_1v_2+v_1v_3+\cdots+v_{n-1}v_n, \\
	&\cdots& \\
	a_n & = & v_1\ldots v_n.
\end{eqnarray*}
is a closed embedding.
\end{lemma}

\begin{proof}
	If $T$ denote the the dual vector space of $V$, then we have $V=\Spec (\Sym(T))$ where $\Sym(T)=\bigoplus_{i=0}^\infty \Sym^i T$ is the symmetric algebra of $T$. We also have $V^n=\Spec(\Sym(T^n))$ and $V^n/\fS_n=\Spec(\Sym(T^n)^{\fS_n})$. On the other hand, we have $\Sym^i V=\Spec(\Sym(\Sym^i T))$. What we have to prove is that the homomorphism of $k$-algebras 
	\begin{equation} \label{symmetric-polynomial}
	\bigotimes_{i=1}^n \Sym(\Sym^i T) \to (\Sym (T^n))^{\fS_n}
\end{equation}  
given by formulae \eqref{chi-coeff}, is surjective for every finite dimensional $k$-vector space $T$. 

When $\dim(T)=1$, this is the fundamental theorem of symmetric polynomials asserting that every symmetric polynomial is a polynomial of the elementary symmetric polynomials.  In one the dimensional case, the homomorphism  \eqref{symmetric-polynomial} is even an isomorphism. This is no longer true in higher dimensional case. The proof of surjectivity in higher dimensional case is completely similar to the usual proof in the one dimensional case. The main difficulty is in fact to choose adequate notations.

The component-wise action of $\Gm^n$ on $T^n$ induces a decomposition in eigenspaces $T^n=T_1\oplus \cdots\oplus T_n$ where $T_1,\ldots,T_n$ are $n$ copies of $T$. This induces an action of $\Gm^n$ on $\Sym(T^n)$ which can also be decomposed as a direct sum of eigenspaces 
$$\Sym(T^n)=\bigoplus_{\nu\in\N^n} \Sym^\nu T$$
where $\Sym^\nu T= \Sym^{\nu_1} T_1\otimes\cdots\otimes \Sym^{\nu_n} T_n$ for each $\nu=(\nu_1,\ldots,\nu_n)\in\N^n$ (we agree that $0\in\N$). 

For every $\nu\in\N^n$, we will denote $[\!\nu]$ its orbit under the symmetric group $\fS_n$. We then have a decomposition 
$$(\Sym(T^n))^{\fS_n}=\bigoplus_{[\!\nu]} \Sym^{[\!\nu]} T$$
where $[\!\nu]$ runs over the set of orbits of $\fS_n$ acting on $\N^n$ and
$$\Sym^{[\!\nu]} T= \left(\bigoplus_{\nu\in\,[\!\nu]} \Sym^{\nu} T\right)^{\fS_n}.$$
In the $\fS_n$-orbit $[\!\nu]$, there exists a unique element $\nu=(\nu_1,\ldots,\nu_n)$ satisfying $\nu_1\geq\ldots \geq\nu_n$. We rewrite $[\!\nu]$ in the partition form 
$$[\!\nu]=[\mu_1^{m_1},\ldots,\mu_e^{m_e}]$$
where $\mu_1 > \cdots > \mu_e \geq 0$ are the distinct integers occurring as coordinates of $(\nu_1,\ldots,\nu_n)$ with multiplicities $m_1,\ldots,m_e$ respectively. We have $m_1+\cdots+m_e=n$ and 
$$\deg(\nu)=\nu_1+\cdots+\nu_n=m_1\mu_1+\cdots+ m_e \mu_e.$$
With these notations being set up, we have
$$\Sym^{[\!\nu]} T = \bigotimes_{i=1}^e \Sym^{m_i}(\Sym^{\mu_i} T).$$
We observe if $\mu_e=0$, the last term $\Sym^{m_e}(\Sym^{\mu_e} T)$ may be removed from this formula. For this reason, we will write $[\!\nu]$ in the more compact form
$$[\!\nu]=[\mu_1^{m_1},\ldots,\mu_e^{m_e}]$$
with $\mu_1>\ldots>m_e >0$, the possible term $\mu_e=0$ being removed. We have then 
$$\lg [\!\nu]:=m_1+\cdots+m_e \leq n.$$
In particular, in degree $d$, we have 
$$(\Sym^d(T^n))^{\fS_n}=\bigoplus_{\deg [\!\nu]=d, \lg [\!\nu] \leq n} \Sym^{[\!\nu]}T.$$
 
In degree 1, we have the obvious isomorphism $T=(T^n)^{\fS_n}$ given by the diagonal map. It starts being tricky already in degree $2$. If $n=1$,
we have the tautological isomorphism $(\Sym^2(T^n))^{\fS_n}=\Sym^2 T$.
If $n\geq 2$, we have the decomposition in degree $2$:
$$(\Sym^2(T^n))^{\fS_n}=\Sym^2 T \oplus \Sym^2 T$$
in which the first copy of $\Sym^2 T$ corresponds to $[\!\nu]=[2]$ whereas the second copy corresponds to $[\!\nu]=[1^2]$.
We also note that for every $i\in\{1,\ldots,n\}$ and $[\!\nu]=[1^i]$, we have an isomorphism 
$$\Sym^i T \simeq \Sym^{[1^i]} T$$
that is exactly the one used in the formulae \eqref{chi-coeff}.

With these notations being set up, we can prove the surjectivity of the algebra homomorphism \eqref{symmetric-polynomial} by the same induction on the degree as in the case $\dim(T)=1$. The induction uses the following fact: For every $i\in\{1,\ldots,n\}$ and 
partition $[\!\nu]=(\nu_1\geq\cdots\geq\nu_n)$, we denote 
	$$[\!\nu+1^i]=(\nu_1+1\geq\cdots \geq \nu_i+1 \geq \nu_{i+1} \cdots\geq \nu_n)$$
	which is a partition of degree $\deg[\!\nu + 1^i]=\deg[\!\nu]+i$. The multiplication map 
	$$\kappa(i,[\!\nu]):\Sym^{[1^i]}T \otimes \Sym^{[\!\nu]}T \to \Sym^{\deg[\!\nu]+i} T$$
	has image contained in 
	$$\bigoplus_{[\!\nu']\leq [\!\nu+1^i]} \Sym^{[\!\nu']}T$$
where $[\!\nu']$ runs over the set of partitions less than $[\!\nu_+\!]$ with respect to the usual partial order among the partitions. More over, the composition of $\kappa_{[\!\nu]}$ with the projection on $\Sym^{[\!\nu_+\!]}$ is surjective. 
This fact can be verified with an easy calculation on tensors. 

We also observe unless $[\nu]=0$, there exists $i\in\{1,\ldots,n\}$ and a partition $[\nu_-]$ such that $[\nu]=[\nu_-+1^i]$. It follows that the vector subspaces $\Sym^{[1^i]}T$ of $(\Sym(T^n))^{\fS_n}$ generate it as an algebra.
\end{proof}

In general, $\Chow_n(V)$ is a strict closed subscheme of $V\times \Sym^2 V \times \cdots \times \Sym^n V$. As an instructive example, let us describe $\Chow_2(\bA^2)$ as a hypersurface of the 5-dimensional affine space $\bA^2\times \Sym^2 \bA^2$. Let $V=\bA^2$ and 
$$\Chow_2(V) \to V\times \Sym^2(V)$$
denote the map given by $[v_1,v_2] \mapsto (v,w)$ with $v=v_1+v_2$ and $w=v_1v_2$. It is clear that $\Chow_2(V)$ is the closed subscheme of $V\times \Sym^2(V)$ consiting of pairs $(v,w)$ such that $v^2-4w$ is of the form $u^2$ with $u\in V$. Now we need to find the equation defining the image of the map $V\to \Sym^2(V)$ given by $u\mapsto u^2$. Now we write $v\in V$ with coordinates $v=(x,y)\in\bA^2$. Then $v^2=(x^2,xy,y^2)$ in $\Sym^2 V=\bA^3$. Now the points $(w_1,w_2,w_3)=(x^2,2xy,y^2)$ move in the quadric defined by the equation $w_2^2-4w_1w_3=0$. It follows that $\Chow_2(V)$ is the closed subscheme of $V\times \Sym^2(V)$ defined by the equation
$$\Chow_2(V)=\{(x,y,w_1,w_2,w_3)\in V\times \Sym^2 V \mid 
(xy-2w_2)^2=(x^2-4w_1)(y^2-4w_3)\}$$
We also note that the morphism $(w_1,w_2,w_3)\mapsto w_2^2-4w_1w_3$ is the explicit expression of the canonical $\GL_2$-equivariant linear map $\Sym^2(\Sym^2 V ) \to (\wedge^2 V)^{\otimes 2}$ whose kernel is $\Sym^4(V)$.

In \cite{Simpson 2}, Simpson constructs the Hitchin map
\begin{equation}\label{Simpson-Hitchin} 
	h_X:\Higgs_X\to \cB_X=\bigoplus_{i=1}^n \rH^0(X,\Sym^i \Omega^1_X)
\end{equation}
assigning $(E,\theta)\in\Higgs_X$ with $h_X(E,\theta)=(b_1,\ldots,b_n)$ where $b_i \in \rH^0(X,\Sym^i \Omega_X^1)$ is the symmetric form defined as follows. By iterating $i$ times the Higgs field $\theta:E\to E\otimes \Omega^1_X$ we get a map $\theta^i:E\to E\otimes \Sym^i \Omega^1_X$, and we set $b_i=\tr(\theta^i) \in \rH^0(X,\Sym^i \Omega^1_X)$. We may also define symmetric form $a_i\in \rH^0(X,\Sym^i \Omega^1_X)$ by using the $\GL_d$-equivariant map
\begin{equation}\label{s_i}
	\Chow_n(\bA^d)\stackrel{\iota}\to \bA^d\times\Sym^2(\bA^d)\times\cdot\cdot\cdot\times\Sym^n(\bA^d)\stackrel{pr}\to \Sym^i(\bA^d)  
\end{equation}
defined as in \eqref{characteristic-polynomial}.
Namely, we can interpret sections in $\rH^0(X,\Sym^i \Omega^1_X)$ as morphisms 
$X\to[\Sym^i(\bA^d)/\GL_d]$ laying over the morphism 
$X\to\GL_d$ corresponding $\Omega_X^1$, and we define 
\[a_i:X\stackrel{a}\to[\Chow_n(\bA^d)/\GL_d]\stackrel{}\to[\Sym^i(\bA^d)/\GL_d]\]
where $a=\sd_X(E,\theta)\in\cA_X$, interpreting 
as a morphism
$a:X\to[\Chow_n(\bA^d)/\GL_d]$ laying over the morphism 
$X\to\GL_d$ corresponding $\Omega_X^1$, 
and 
the second arrow is the morphism induced by (\ref{s_i}).
The variables $(a_i)$ and $(b_i)$ can be obtained from each other by the formula
$$ia_i=-(b_i+a_1b_{i-1}+\cdot\cdot\cdot+a_{i-1}b_1)$$ 
valid for all $i=1,...,n$. In other words, the change of variables from $(a_i)$ to $(b_i)$ and vice versa is given by an automorphism of $\prod_{i=1}^n \Sym^i(\bA^d)$. It gives rise to a morphism 
\begin{equation} \label{iota_X}
	\iota_X:\cA_X\to \cB_X
\end{equation}
such that
\begin{equation}\label{factorization-Hitchin}
	h_X=\iota_X \circ \sd_X
	\end{equation}
By lemma \ref{closed-immersion}, we know that $\iota_X$ is a closed immersion.
This provides a canonical factorization of the Hitchin map though the spectral data morphism.

\begin{remark}
One can regard
$\iota_X$
as a global version 
of the embedding $\iota$ in (\ref{characteristic-polynomial}).
\end{remark}

It seems natural to conjecture that $\cA_X$ is the image of $h_X$ in $\cB_X$. We don't know it for the moment but it appears that the geometry of $\cA_X$ may be very complicated (see, e.g., section \ref{Ab surfaces}). In the case of surfaces, we will construct an open subset 
$\cA^\heartsuit_X$ of $\cA_X$
such that for every geometric point $a\in \cA^\heartsuit_X$ we can describe the fiber $\sd_X^{-1}(a)$ im a way similar to \cite{BNR} and prove that $\sd^{-1}_X(a)\neq\emptyset$. 

\section{Universal spectral cover} 

In \cite{Hitchin}, Hitchin described fibers of his fibration with help of the concept of spectral curves. We are now about to generalize his construction to the case of higher dimensional varieties and construct spectral cover attached to each spectral datum. 

The construction of spectral covers relies on the Cayley-Hamilton equation. Let $V$ be a $d$-dimensional vector space, $\Chow_n(V)$ the Chow scheme of $0$-cycles of length $n$ on $V$. We will represent a point of $\Chow_n(V)$ as an unordered collection of $n$ points of $V$
\begin{equation}
	[v_1,\ldots,v_n]\in \Chow_n(V).
\end{equation} 
Recall that we have a closed embedding $\Chow_n(V)\to V\times \Sym^2 V \times \cdots \times \Sym^n V$ given by $$[v_1,\ldots,v _n]\mapsto (a_1,\ldots,a_n)$$ 
as in \eqref{chi-coeff}.
We consider the morphism
\begin{equation}
	\chi_V:V\times \Chow_n(V) \to \Sym^n V
\end{equation}
given by 
\begin{equation}
\chi_V(v,[v_1,\ldots,v_n])= (v-v_1)\ldots(v-v_n)=v^n-a_1 v^{n-1}+\cdots+(-1)^n a_n.	
\end{equation}
We define the closed subscheme $\on{Cayley}_n(V)$ to be 
\begin{equation}
	\on{Cayley}_n(V)=\chi_V^{-1}(\{0\})
\end{equation}
the fiber of $0\in \Sym^n V$. 

\begin{lemma} \label{Cayley}
Let $V$ be a $d$-dimensional vector space.
\begin{enumerate}
\item The projection of $p_n(V):\on{Cayley}_n(V)\to\Chow_n(V)$ is a finite morphism which is étale over the open subset $\Chow'_n(V)$ of $\Chow_n(X)$ consisting of multiplicity free $0$-cycles. We call $\on{Cayley}_n(V)$ the universal spectral cover of $\Chow_n(V)$. 
\item  For every point $a=[x_1^{n_1},\ldots,x_m^{n_m}]\in \Chow_n(V)$ where $x_1,\ldots,x_m$ are distinct points of $V$, and $n_1,\ldots,n_m$ are positive integers such that $n_1+\cdots+n_m=n$, the fiber of $\on{Cayley}_n(V)$ over $a$ is the finite subscheme of $V$
	\begin{equation}
		\on{Cayley}_n(V,a)=\bigsqcup_{i=1}^m \Spec(\cO_{V,x_i}/\fm_{x_i}^{n_i}),
	\end{equation}
where $\cO_{V,x_i}$ is the local ring of $V$ at $x_i$, and $\fm_{x_i}$ its maximal ideal. In particular, as soon as $d\geq 2$ and $n\geq 2$, then the cover $\on{Cayley}_n(V)\to \Chow_n(V)$ is not flat.

\item Let $F$ be a finite $\cO_V$-module of length $n$ and $a\in\Chow_n(V)$ its spectral datum. Then $F$ is supported by the finite subscheme $\on{Cayley}_n(V,a)$ of $V$. (This is a generalization of the Cayley-Hamilton theorem)
\end{enumerate}	
\end{lemma}

\begin{proof}
We will first describe a set of the generators of the ideal defining the closed subscheme $\on{Cayley}_n(V)$ in $V\times\Chow_n(V)$. Each linear form $t:V\to k$ induces a map on Chow varieties $[t]:\Chow_n(V)\to \Chow_n(\bA^1)$ mapping $a=[v_1,\ldots,v_n]\in\Chow_n(V)$ on $t(a)=[t(v_1),\ldots,t(v_n)]\in \Chow_n(\bA^1)$. For the diagram 
\begin{equation}
	\begin{tikzcd}
\Chow_n(V) \times V \arrow{r}{\chi_V} \arrow{d}[swap]{[t]\times t}
& \Sym^n(V) \arrow{d}{\Sym^n(t)} \\
\Chow_n(\bA^1) \times \bA^1 \arrow{r}[swap]{\chi_{\bA^1}}
& \Sym^n(\bA^1)=\bA^1
\end{tikzcd}
\end{equation}
is commutative, the function $f_t=\chi_{\bA^1}\circ (t\times [t]):\Chow_n(V) \times V \to \bA^1$ vanishes on $\on{Cayley}_n(V)$. Explicitly for every $a=[v_1,\ldots,v_n]\in \Chow_v(V)$, we have
\begin{equation} \label{fta}
	f_t(a,v)=(t(v)-t(v_1)) \ldots (t(v)-t(v_n)).
\end{equation}
Moreover, for $\Sym^n(t)$ generates the ideal defining $0$ in $\Sym^n(V)$ as $t$ varies on the dual vector space $T$ of $V$, the functions $f_t$ generate the ideal defining $\on{Cayley}_n(V)$ inside $\Chow_n(V)\times V$. This provides a convenient set of generators of this ideal albeit infinite and even innumerable as $k$ may be. 

\begin{enumerate}
\item If $t_1,\ldots,t_d$ form a basis of $T$, then $f_{t_1},\ldots,f_{t_d}$ cut out a closed  subscheme $Z$ of $\Chow_n(V)\times V$ which is finite flat of degree $n^d$ over $\Chow_n(V)$. Since $\on{Cayley}_n(V)$ is a closed subscheme of $Z$, it is also finite over $\Chow_n(V)$. This proves the first assertion of the lemma.

\item We will prove that for $a=[x_1^{n_1},\ldots,x_m^{n_m}]\in \Chow_n(V)$ where $x_1,\ldots,x_m$ are distinct points of $V$, and $n_1,\ldots,n_m$ are positive integers such that $n_1+\cdots+n_m=n$, $\on{Cayley}_n(V,a)$ is the closed subscheme of $V$ defined by the ideal $\fm_{x_1}^{n_1} \ldots \fm_{x_m}^{n_m}$ of $k[V]$ where $\fm_{x_i}$ is the maximal ideal corresponding to the point $x_i\in V$. 

Let us denote $I_a$ the ideal of $A=k[V]$ defining the finite subscheme $\on{Cayley}_n(V,a)$ in $V$. We first prove that $I=I_{x_1}\ldots I_{x_n}$ where $A/I_{x_i}$ is supported by some finite thickening of the point $x_i$. For this we only need to prove that for every $x\notin\{x_1,\ldots,x_m\}$, there exists a function $f\in I_a$ such that $f\notin\fm_x$. We recall that the ideal $I_a$ is generated by thefunctions $f_t(a):V\to\bA^1$ as $t$ varies in $V^*$. Choose a linear form $t:V\to k$ such that $t(x)\neq t(x_i)$ for all $i\in\{1,\ldots,m\}$, then we have $f_t(a)(x)\neq 0$ by \eqref{fta}. 

As $x_1,\ldots,x_n$ play equivalent roles, it  only remains to prove that the images of the functions $f_t(a)$ in the localization $A_{x_1}$ of $A$ at $x_1$, as $t$ varies in $V^*$, generate the ideal $\fm_{x_1}^{n_1}$. From \eqref{fta}, we already know that $f_t(a)\in \fm_{x_1}^{n_1}$ for every $t\in V^*$. By the Nakayama lemma, we only need to prove that the images of $f_t(a)$ in $\fm_{x_1}^{n_1}/\fm_{x_1}^{n_1+1}$ generate this vector space as $t$ varies in $V^*$. We observe that for $t\in V^*$ such that $t(x_1)\neq t(x_i)$ for $i\in \{2,\ldots,m\}$, the factors $t(v)-t(x_2),\ldots t(v)-t(v_m)$ are all invertible at $x_1$, it is enough to prove that for $T\in V^*$ satisfying the open condition $t(x_1)\neq t(x_i)$ for $i\in \{2,\ldots,m\}$, the functions $(t(v)-t(x_1))^{n_1}$ generate $\fm_{x_1}^{n_1}/\fm_{x_1}^{n_1+1}$. Here we use again the fact the image of the $n$-th power map $\fm_x/\fm_x^2 \to \fm_x^n/\fm_x^{n+1}$ span $\fm_x^n/\fm_x^{n+1}$ and this conclusion doesn't change even after we remove from $\fm_x/\fm_x^2$ a closed subset of smaller dimension.

\item By the Chinese remainder theorem we are easily reduced to prove that if $E$ is a finite $A$-module of length $n$, supported by a finite thickening of $x\in V$ then $E$ is annihilated by $\fm_x^n$. Since $E$ is  supported by a finite thickening of $x\in V$ it has a structure of $A_x$-module where $A_x$ is the localization of $A$ at $x$. We consider the decreasing filtration $E \supset \fm_x E \supset \fm_x^2 E \supset \cdots$.  By the Nakayama lemma, we know that for $m\in\N$, $\fm_x^m E/\fm_x^{m+1}E=0$ implies $\fm_x^m E=0$. It follows that as long as $\fm_x^m E\neq 0$, we have $\dim_k(\fm_x^i E/\fm_x^{i+1}E)\geq 1$ for all $i\in \{0,\ldots,m\}$ and it follows that $m+1\leq n$. We conclude that $\fm_x^n E=0$. 
\end{enumerate}
This completes the proof of lemma \ref{Cayley}
\end{proof}

The morphism $p_n(V):\on{Cayley}_n(V)\to \Chow_n(V)$ being $\GL(V)$-equivariant, it can be twisted by any rank $d$ vector bundle. In particular, if  $X$ be a $d$-dimensional smooth variety over $k$, we have a finite morphism $p_n(\rT^*_X/X):\on{Cayley}_n(\rT^*_X/X)\to \Chow_n(\rT^*_X/X)$. For every spectral datum $a\in \cA_X$ corresponding to a section $a:X\to \Chow_n(\rT^*_X/X)$, we can pull back $p_n(\rT^*_X/X)$ to obtain the spectral cover
\begin{equation}
	p_a:Y_a \to X.
\end{equation}
Just as $p_n(\rT^*_X/X)$, the morphism $p_a:Y_a \to X$ is finite but in general not flat. The morphism $p_a:Y_a \to X$ factors through a canonical closed embedding $Y_a\to \rT^*_X$.

We will be specially interested in the open subset $\cA^\heartsuit_X$ of $\cA_X$ consisting of maps $a:X\to \Chow_n(\rT^*_X/X)$ that maps the generic point of $X$ to the open subset $\Chow_n'$ of multiplicity free 0-cycles. If $a\in\cA_X^\heartsuit$, the covering $p_a:Y_a\to X$ is generically finite étale of degree $n$.

If $X$ is a curve, and if the spectral curve $Y_a$ is integral, after Beauville-Narasimhan-Ramanan \cite{BNR}, there is an equivalence of categories between the category of Higgs bundles $(E,\theta)$ of spectral datum $a$ and the category of torsion-free $\cO_{\tilde X_a}$ of generic rank 1. This equivalence can be generalized to the case $d\geq 1$ with aid of the concept of Cohen-Macaulay sheaves. 

Let us recall some basic facts about Cohen-Macaulay sheaves. A coherent sheaf $M$ on a purely $n$-dimensional scheme $Y$  is said to be \emph{maximal Cohen-Macaulay} if $\dim(\Supp(M))=\dim(Y)$ and $\rH^i(\DD(M))=0$ for $i\neq n$. A family of maximal Cohen-Macaulay sheaves on $Y$ parametrized by $S$ is a coherent sheaf $M$ on $Y\times S$  flat over $S$ and satisfying the property: for every $s\in S$, the restriction $M_s$ to $Y\times \{s\}$ is maximal Cohen-Macaulay. If $S$ is Cohen-Macaulay, then the above conditions imply that $M$ itself is maximal Cohen-Macaulay. If $Y$ is proper, the functor that associates with every test scheme $S$ the groupoid of all families of Cohen-Macaulay sheaves on $Y$ parametrized by $S$ is an algebraic stack, see \cite[2.1]{Arinkin}. 

We also recall an important fact about Cohen-Macaulay module. Let $M$ be a Cohen-Macaulay $R$-module of finite type. Suppose that $R$ is a finite $A$-algebra with $A$ being a regular ring. Then $M$ is a locally free $A$-module of finite type. We refer to \cite[section 2]{BBG} for a nice discussion on Cohen-Macaulay modules and for further references therein, or the comprehensive treatment in \cite{BH}.

\begin{proposition}
	For every $a\in\cA^\heartsuit_X$, the Hitchin fiber $\Higgs_{X,a}=\on{sd}_X^{-1}(a)$ is isomorphic to the algebraic stack of maximal Cohen-Macaulay sheaves of generic rank one on the spectral cover $Y_a$.
\end{proposition}

\begin{proof}
	Let $(E,\theta)\in\Higgs_{X,a}$ a Higgs bundle of rank $n$ whose spectral datum is $a\in\cA^\heartsuit_X$. Then $E=p_* F$ where $F$ is a coherent sheaf over the cotangent $\rT^*_X$. By the Cayley-Hamilton theorem, $F$ is supported by the spectral cover $Y_a \subset \rT^*_X$. We have then $E=p_{a*} F$ where $F$ is a coherent sheaf on $Y_a$. Since $p_a:Y_a\to X$ is a finite morphism, and $E$ is a vector bundle over $X$, $F$ is a maximal Cohen-Macaulay sheaf. Moreover, since $p_a$ is generically finite étale of degree $n$, $F$ is of generic rank one. Inversely, if $F$ is a maximal Cohen-Macaulay sheaf of generic rank one over $Y_a$, then $E=p_{a*}F$ is a vector bundle of rank $n$ over $X$. It is naturally equipped with a Higgs field $\theta:E \otimes_{\cO_X}\cT_X \to E$ as $Y_a$ is a closed subscheme of $\rT^*_X$.
\end{proof}

In spite of the simplicity of the above description of $\Higgs_{X,a}$, it is not of great use by itself alone. For instance, it doesn't imply that $\Higgs_{X,a}$ is non empty. The difficulty is that in general for the spectral cover $Y_a$ is not Cohen-Macaulay, we do not know any recipe to construct coherent Cohen-Macaulay sheaves on $Y_a$. At this point, we see that in order to obtain a useful description of $\Higgs_{X,a}$, one needs to construct a finite Cohen-Macaulayfication of $Y_a$. This can be done in the surface case.

\section{Cohen-Macaulay spectral surfaces}

In the surface case, for every $a\in\cA_X^\heartsuit$, the spectral surface $Y_a$ admits a canonical finite Cohen-Macaulayfication
whose construction will rely on the Hilbert schemes of points on surfaces and Serre's theorem on extending vector bundles on smooth surfaces across a closed subscheme of codimension two. 

We first recall some well known fact about the Hilbert schemes of $0$-dimensional subschemes of a surface.
Let $\Hilb_n(\bA^2)$ denote the moduli space of zero-dimensional subschemes of length $n$ of $\bA^2$. A point of $\Hilb_n(\bA^2)$ is a $0$-dimensional subscheme $Z$ of $\bA^2$ of length $n$ that will be of the form $Z=\bigsqcup_{\alpha \in\bA^2} Z_\alpha$ where $Z_\alpha$ is a local $0$-dimensional subscheme of $\bA^2$ whose closed point is $\alpha$. It is known that the map 
\begin{equation} \label{Hilbert-Chow}
{\rm HC}_n:\Hilb_n(\bA^2) \to \Chow_n(\bA^2).
\end{equation}
given by $Z\mapsto \sum_{\alpha\in\bA^2} \length(Z_\alpha)\alpha$, where $\length(Z_\alpha)$ is the length of $Z_\alpha$, is a resolution of singularities of $\Chow_n$. If $\Chow'_n$ denotes the open subset of $\Chow_n$ classifying multiplicity-free $0$-cycle of length $r$ on $\bA^2$ then ${\rm HC}_n$ is an isomorphism over $\Chow'_n$. The complement $Q$ of $\Chow'_n$ in $\Chow_n$ is a closed subset of codimension 2.

As the morphism \eqref{Hilbert-Chow} is $\GL_2$-equivariant, we can twist it by any $\GL_2$-bundle, and in particular by the $\GL_2$-bundle associated to the cotangent bundle $\rT^*_X$ over a smooth surface $X$ and by doing so we obtain
\begin{equation}\label{HC-Omega}
	{\rm HC}_{\rT^*_X/X}:\Hilb_n(\rT^*_X/X)\to \Chow_n(\rT^*_X/X).
\end{equation}
This morphism is a proper morphism and its restriction the open subset $\Chow'_n(\rT^*_X/X)$ is an isomorphism. Here the open immersion $\Chow'_n(\rT^*_X/X) \subset \Chow_n(\rT^*_X/X)$ is obtained from the open subscheme $\Chow'_n$ of $\Chow_n$ classifying multiplicity free $0$-cycles by the process of $\GL_2$-twisting by the cotangent bundle. 

We will now recall Serre's theorem on extending locally free sheaves across a closed subscheme of codimension 2, see \cite[Prop. 7]{Serre}.

\begin{theorem}
	Let $X$ be a smooth surface over $k$, $Z$ a closed subscheme of codimension 2 of $X$ and $j:U\to X$ the open immersion of the complement $U$ of $Z$ in $X$. Then the functor $V\to j_*V$ is an equivalence of categories between the category of locally free sheaves on $U$ and locally free sheaves on $X$. Its inverse is the functor $j^*$. 
\end{theorem}

We recall that a point $a\in \cA_X$ in the space of spectral data is a section $a:X\to \Chow_n(\rT^*_X/X)$. We have also defined $\cA^\heartsuit_X$ to be the open subset of $\cA_X$ of sections $a: X\to \Chow_n(\rT^*_X/X)$ mapping the generic point of $X$ to the open subset $\Chow'_n(\rT^*_X/X)$, in other words
\begin{equation} \label{A-heart}
	\cA^\heartsuit_X=\{a\in\cA_X| \dim a^{-1}(Q_X)\leq 1\}
\end{equation}
where $Q_X$ is the complement of $\Chow'_n(\rT^*_X/X)$ in  $\Chow_n(\rT^*_X/X)$.

\begin{proposition}\label{CM-fication}
	For every $a\in \cA^\heartsuit_X$, there exists a unique finite flat covering 
	\begin{equation}\label{CM surfaces}
		\tilde p_a:\tilde Y_a\to X
	\end{equation}
	 of degree $n$, equipped with a $X$-morphism $\iota_a:\tilde Y_a \to \rT^*_X$ satisfying the following property: there exists an open subset $U\subset X$, whose complement is a closed subset of codimension at least 2, such that $\iota_a$ is a closed embedding over $U$ and for every $x\in U$, the fiber $\tilde Y_a(x)$ is a point of $\Hilb_n(\rT^*_X/X)(x)$ laying over the point $a(x)\in \Chow_n(\rT^*_X/X)(x)$. Moreover, the morphism $\iota: \tilde Y_a\to \rT^*_X$ factors through the closed subscheme $Y_a$ of $\rT^*_X$ and the implied morphism $\tilde Y_a\to Y_a$ is a finite Cohen-Macaulayfication of $Y_a$.
\end{proposition}

\begin{proof}
	Let $U'$ be the preimage of $\Chow'_n(\rT^*_X/X)$ by the section $a:X\to \Chow_n(\rT^*_X/X)(x)$. By assumption $a\in \cA^\heartsuit$, $U'$ is a non empty open subset of $X$. As the morphism $\HC_{\rT^*_X/X}$ of \eqref{HC-Omega} is an isomorphism over $\Chow_n(\rT^*_X/X)$, we have a unique lifting 
	$$b':U'\to \Hilb_n(\rT^*_X/X)\times_X U'$$ 
	laying over the restriction $a$ to $U'$. 
	
	Since the Hilbert-Chow morphism \eqref{HC-Omega} is proper, there exists an open subset $U \subset X$, larger than $U'$, whose complement $X-U$ is a closed subscheme of codimension at least 2, such that $b':U'\to \Hilb'_n(\rT^*_X/X)\times_X U'$ extends to 
	$$b_U:U\to \Hilb_n(\rT^*_X/X)\times_X U.$$
	By pulling back from $\Hilb_n(\rT^*_X/X)$ the tautological family of subschemes of $\rT^*_X$, we get a finite flat morphism $\tilde U_a \to U$ of degree $n$, equipped with a closed embedding $\iota_U:\tilde U_a \to \Omega_{X}\times_X U$.  
	
	According to Serre's theorem on extending vector bundles over surfaces, there exists a unique the finite flat covering $\tilde Y_a\to X$ of degree $n$ extending the covering $\tilde U_a$ of $U$. The closed embedding $\tilde U_a\to \Omega_{X}\times_X U$ extends to a finite morphism $\tilde Y_a \to \rT^*_X$ which may no longer be a closed embedding.
	
	For by construction $\tilde p_a:\tilde Y_a \to X$ is a finite flat morphism of degree $n$, $\tilde Y_a$ is Cohen-Macaulay. Apply the generalized Cayley-Hamilton theorem to the vector bundle $\tilde p_{a*}\cO_{\tilde Y_a}$, as $\cO$-module over $\rT^*_X$, it is supported by $Y_a$. It follows that the morphism $\tilde Y_a\to \rT^*_X$ factors through $Y_a$. Since $\tilde Y_a$ is finite over $X$, it is also finite over $Y_a$. As the finite morphism $\tilde Y_a\to Y_a$ is an isomorphism over the nonempty open subset $U'$, it is a finite Cohen-Macaulayfication of $Y_a$.
\end{proof}

Instead of using the Hilbert scheme, we can construct $Y_a$ over the height one point as follows. Let $U'=a^{-1}(\Chow'_n(\rT^*_X))$ and $Z'$ le complement of $U'$. Let $z$ be the generic point of an irreducible component of $Z'$. The localization of $X$ at $x$ is $X_z=\Spec(\cO_{X,z})$ where $\cO_{X,z}$ is a discrete valuation ring. By restricting $p_{a*} \cO_{Y_a}$ to $\cO_{X,z}$ we get a module of finite type which may have torsion. By considering the quotient $\Spec(p_{a*} \cO_{Y_a}/(p_{a*} \cO_{Y_a})^{\rm tors})$ we obtain a section $X_z\to \Hilb_n(\rT^*_X/X)\times_X X_z$ over $a|_{X_z}$ and by uniqueness of such a section we have:
\begin{equation}\label{torsion-free-quotient}
	\Spec(p_{a*} \cO_{Y_a}/(p_{a*} \cO_{Y_a})^{\rm tors})= \Spec(\tilde p_{a*} \cO_{\tilde Y_a}).
\end{equation}

\begin{proposition}\label{Hitchin fibers}
For every $a\in \cA^\heartsuit_X$, the fiber $\Higgs_{X,a}$ consists in the moduli stack of Cohen-Macaulay sheaves $F$ over the Cohen-Macaulay spectral surface $\tilde Y_a$  of generic rank one. It contains in particular the Picard stack $\cP_a$ of line bundles on $\tilde Y_a$. The action of $\cP_a$ on itself by translation extends to an action of $\cP_a$ on $\Higgs_{X,a}$.
\end{proposition}

\begin{proof}
Let $(V,\theta)\in\Higgs_X$ be a Higgs bundle over $X$ whose spectral datum is $a\in\cA_X^\heartsuit$. The Higgs field $\theta:V\otimes\cT_{X}\to V$ define a homomorphism $\Sym_X(\cT_X)\to \End_X(V)$ which factors through $p_{a*} \cO_{Y_a}$ by the generalized Cayley-Hamilton theorem. 

Let $U'=a^{-1}(\Chow'_n(\rT^*_X))$ and $Z'$ le complement of $U'$. Let $z$ be the generic point of an irreducible component of $Z'$. The localization of $X$ at $x$ is $X_z=\Spec(\cO_{X,z})$ where $\cO_{X,z}$ is a discrete valuation ring. Over $X_z$ we have a homomorphism
$$p_{a*} \cO_{Y_a}\otimes \cO_{X_z} \to \End_{X}(V)\otimes \cO_{X_z}.$$
Since the target is clearly torsion free, this homomorphism factors through \eqref{torsion-free-quotient}. Thus over an open subset $U\subset X$ whose complement is of codimension two, the above morphism factors through a homomorphism of algebras
$$\tilde p_{a*} \cO_{\tilde Y_a} \otimes \cO_U \to \End_{X}(V)\otimes \cO_{U}.$$
By applying Serre's theorem again, we have a canonical homomorphism
$$\tilde p_{a*} \cO_{\tilde Y_a} \to \End_X(V)$$
so that $V=\tilde p_{a*} F$ where $F$ is a Cohen-Macaulay $\cO_{\tilde Y_a}$-module of generic rank one. 

Since $\tilde p_a:\tilde Y_a\to X$ is finite and flat, for every line bundle $L$ on $\tilde Y_a$, $p_{a*} L$ is a vector bundle of rank $n$ carrying a Higgs field. Thus $\cP_a \subset \Higgs_{X,a}$. We have an action of $\cP_a$ on $\Higgs_{X,a}$ defined by $(L,F) \mapsto L\otimes_{\cO_{\tilde Y_a}} F$ where $L$ is a line bundle on $\tilde Y_a$ and $F$ is a Cohen-Macaulay sheaf of generic rank one.
\end{proof}

\begin{remark}
Let $a\in\sA_X^\heartsuit$ such that the Cohen-Macaulay surface $\tilde Y_a$ is integral.
According to \cite{AK},
the moduli space $\on{Pic}(\tilde Y_a)^{-}$
of 
Cohen-Macaulay sheaves on $\tilde Y_a$ of generic rank one (which by the proposition above is the same as the rigidified Hitchin fiber over $a$) admits a compactification $\Pic(\tilde Y_a)^{=}$ given by the moduli space of 
of torsion free rank 
one sheaves on $\tilde Y_a$. On the other hand, 
by \cite[Theorem 6.11]{Simpson 2},  
the (extended) Hitchin map $h_X^{=}:\on{Higgs}^{=}_X\to\sB_X$ from 
$\on{Higgs}^{=}_X$, the moduli space  of torsion free Higgs sheaves on $X$, to 
$\sB_X$
is proper. One can show that the fiber $(h_X^{=})^{-1}(a)$ is isomorphic to 
$\Pic(\tilde Y_a)^{=}$.
\end{remark}

\begin{proposition}
	Let $a\in\cA^\heartsuit_X$ such that the surface $\tilde Y_a$ is normal, the reduced rigidified neutral component 
$(\cP^0_{a})_{\on{red}}$ of $\cP_a$ is a quotient of abelian variety by $\Gm$ acting trivially.\footnote{
If the characteristic of $k$ is zero, then $\cP_a$ is smooth and
the reducedness assumption is automatic.}
\end{proposition}

\begin{proof}
	This is a consequence of a theorem of Geisser \cite[Theorem 1]{Geisser}. Geisser's theorem states that the multiplicative part the reduced Picard variety $P_a^0$ of an algebraic variety $Y$ is trivial if and only if $\rH^1_{et}(Y,\Z)$ is trivial whereas the unipotent part is trivial if and only $Y$ is seminormal. If $Y$ is normal, $\pi_1(Y)$ is a profinite group, quotient of the Galois group of the generic point, and therefore can't afford a nontrivial continuous homomorphism to $\Z$. It follows that $\rH^1(Y,\Z)$ is trivial. On the other hand, a normal variety is certainly also seminormal. Now after Geisser, $P_a^0$ is an abelian variety. We have $(\cP_a^0)_{\on{red}}=[P_a^0/\Gm]$.
\end{proof}

\section{Abelian surfaces}\label{Ab surfaces}
We study the Hitchin fibration for abelian surfaces. Let $X$ 
be an abelian surface, that is,
a two dimensional abelian variety. 
Since $X$ is an algebraic group, we have a canonical trivialization of 
the cotangent bundle $\rT^*_X\cong X\times V$, where $V=\rT^*_X|_e$ is the fiber of $\rT^*_X$
at the 
identity $e\in X$. 
The relative Chow variety 
$\on{Chow}_n(\rT^*_X/X)$ can be identified with $X\times\on{Chow}_n(V)$ and, since 
both 
$\on{Chow}_n(V)$ and $V\times\Sym^2V\times\cdot\cdot\cdot\times
\Sym^nV$
are affine, we have 
\beq\label{iso with chow_n}
\mathscr A_X\is\on{Chow}_n(V),\ \ \mathscr A_X^\heartsuit\is\on{Chow}'_n(V),
\ \ \mathscr B_X\is V\times\Sym^2V\times\cdot\cdot\cdot\times
\Sym^nV.
\eeq
In particular, we see that 
$\mathscr A_X^\heartsuit$ is open dense in $\mathscr A_X$ and 
$\mathscr A_X\subsetneq\mathscr B_X$ 
is a proper subset 
for $n\geq 2$.

Under the above isomorphisms
the spectral data map and Hitchin map morphism become
\[\on{sd_X}:\on{Higgs}_X\to\on{Chow}_n(V),\ \ h_X:\on{Higgs}_X\to V\times\Sym^2V\times\cdot\cdot\cdot\times
\Sym^nV\]
and 
the factorization of the Hitchin map in (\ref{factorization-Hitchin}) becomes
\[h_X:\on{Higgs}_X\stackrel{\on{sd}_X}\to
\on{Chow}_n(V)\stackrel{\iota}\to
V\times\Sym^2V\times\cdot\cdot\cdot\times
\Sym^nV.\]
Here $\iota$ is the closed embedding in lemma \ref{closed-immersion}.

We claim that the spectral data map $\on{sd_X}$ is surjective. Let 
$a\in\mathscr A_X=\on{Chow}_n(V)$. Choose a point 
$Z=\bigsqcup_{\alpha\in V}Z_\alpha\in\on{Hilb}_n(V)$ 
such that $\on{HC}_n(Z)=a$, here $Z_\alpha$ is a local zero dimensional subscheme
of $V$ whose closed point is $\alpha$.
Consider the 
closed subscheme
\[X\times Z=\bigsqcup_{\alpha\in V}X\times Z_\alpha\subset\rT_X^*=X\times V.\]
Then the rank $n$ bundle $E:=pr_{X*}\mO_{X\times Z}$ ($pr_X:\rT^*_X\to X$) is naturally 
equipped with a Higgs field $\theta$ and it follows from the construction that the spectral datum of the Higgs bundle $(E,\theta)$ 
is equal to $a$. The claim follows.

Let $a=[x_1,...,x_n]\in\mathscr A_X^\heartsuit=\on{Chow}'_n(V)$. The 
associated 
spectral surface $Y_a$ is given by
\[Y_a=\bigsqcup_{i=1}^n X\times\{x_i\}\subset X\times V\is\rT_X^*.\]
In particular, $Y_a$ is smooth and is isomorphic to 
the finite Cohen-Macaulayfication $\tilde Y_a$. Since $Y_a$ is smooth, Cohen-Macaulay sheaves on
$Y_a$ are locally free and it follows from proposition \ref{Hitchin fibers} that 
the fiber $\on{Higgs}_{X,a}$ is isomorphic to 
\[\on{Higgs}_{X,a}\is\sP ic(Y_a)\is
\sP ic(X\times\{x_1\})\times\cdot\cdot\cdot\times\sP ic(X\times\{x_n\}).\]

\begin{remark}
The discussion above can be easily generalized to higher dimensional 
abelian varieties.
\end{remark}

\section{Surfaces fibered over a curve}
In this section we study the spectral surfaces $Y_a$ and the Cohen-Macaulay spectral surface $\tilde Y_a$ in the case when $X$ is a fibration over a 
curve $C$. The results of this section will be used later in the study of 
Hitchin fibration for ruled and elliptic surfaces.

Let $X$ be a smooth projective surface and let 
$C$ be a smooth projective curve. Assume there is a proper flat 
surjective map $\pi:X\to C$ such that the generic fiber a smooth projective curve.
We denote by $X^0\subset X$ the largest open subset such that $\pi$ is smooth.

Consider the cotangent morphism $d\pi:\rT_C^*\times_CX\to \rT_X^*$.
It induces a map \[[d\pi]:\on{Chow}_n(\rT_C^*/C)\times_CX
\to\on{Chow}_n(\rT_X^*/X)\]
on the relative Chow varieties. 
For every section 
$a_C:C\to\on{Chow}_n(\rT_C^*/C)$, the composition 
$$a_X:X\stackrel{}\is C\times_CX\stackrel{a_C\times\on{id}_X}\to\on{Chow}_n(\rT_C^*/C)\times_CX\stackrel{[d\pi]}\to\on{Chow}_n(\rT_X^*/X)$$
is a section of $\on{Chow}_n(\rT_X^*/X)\to X$ and the assignment $a_C\to a_X$
defines a map on the spaces of spectral data
\beq\label{embedding}
\iota_\pi:\mathscr A_C\to\mathscr A_X.
\eeq
We claim that the map above is an embedding. To see this we observe that 
there is a commutative diagram 
\beq\label{maps between bases}
\xymatrix{\mathscr A_C\ar[r]^{\iota_\pi}\ar[d]^{\iota_C}&\mathscr A_X\ar[d]^{\iota_X}\\
\mathscr B_C\ar[r]^{j_\pi}&\mathscr B_X}
\eeq
where the vertical arrows are the embeddings in (\ref{factorization-Hitchin}),
and the bottom arrow is the embedding 
\[j_\pi:\mathscr B_C=\bigoplus_{i=1}^n\rH^0(C,S^i\Omega_C^{1})\hookrightarrow\mathscr B_{X}=
\bigoplus_{i=1}^n\rH^0(X,S^i\Omega^1_X)\]
induced by the  injection 
$\rH^0(C,S^i\Omega_C^1)= 
\rH^0(X,\pi^*S^i\Omega_C^1)\to\rH^0(X,S^i\Omega_X^1)$. 
The claim follows. 
Note that, since $\on{dim} C=1$, the left vertical arrow in (\ref{maps between bases}) is in fact an isomorphism. 
From now on we will view $\mathscr A_C$ 
as subspace of $\mathscr A_X$.

Since the cotangent map 
$d\pi:\rT_C^*\times_CX\to\rT_X^*$ is a closed imbedding over the 
open locus $X^0$, we have $\mathscr A_C^\heartsuit=\mathscr A_C\cap\mathscr A_X^\heartsuit$. 

For any $a\in\mathscr A_C$,
we denote by 
$C_a\to C$ the corresponding spectral curve and we define  
$\tilde X_a=C_a\times_CX$. 
The natural projection
map $\pi_a:\tX_a\to X$ is 
finite flat of degree $n$ and it implies $\tilde X_a$ is a Cohen-Macaulay surface.

\begin{lemma}\label{Y_a}
There exits a finite $X$-morphism
$\tilde\pi_a:\tilde X_a\to Y_a$ which is a generic isomorphism
if $a\in\mathscr A_C^\heartsuit$. If the fibration $\pi:X\to C$ has only reduced fibers, then 
for any $a\in\mathscr A_C^\heartsuit$, the map 
$\tilde\pi_a:\tilde X_a\to Y_a$ is isomorphic to the 
finite Cohen-Macaulayfication
$\tilde p_a:\tilde Y_a\to Y_a$ in (\ref{CM surfaces}) (which is well-defined since $a\in\mathscr A_X^\heartsuit$).

\end{lemma}

\begin{proof}
Let
$i_a:\tilde X_a\to \rT^*_X$ be the restriction of 
the cotangent morphism $d\pi:\rT_C^*\times_CX\to \rT^*_X$
to the closed sub-scheme $\tilde X_a\subset \rT_C^*\times_CX$. By the 
Cayley-Hamilton theorem the map $i_a$ factors through 
the spectral surface $Y_a$. Let 
$\tilde\pi_a:\tilde X_a\to Y_a$ be the 
resulting map. 
As $\tX_a$ is finite over $X$, the map $\tilde\pi_a$ is finite.
In addition, if $a\in\mathscr A_C^\heartsuit$,
then both $\tilde X_a$ and $Y_a$ are generically \'etale over 
$X$ of degree $n$ and it implies $\tilde\pi_a$ is a generic isomorphism.
 Assume the fibers of $\pi$ are reduced.
Then the smooth locus $X^0$ of the map $\pi$ is open and its complement 
$X-X^0$ is a closed subset of codimension $\geq 2$.
Since the map $i_a:\tilde X_a\to \rT^*_X$
is a closed embedding over $X^0$, 
proposition \ref{CM-fication} implies 
the finite flat covering   
$\tilde\pi_a:\tilde X_a\to X$ 
is isomorphic to the finite Cohen-Macaulayfication $\tilde p_a:\tilde Y_a\to X$. 
\end{proof}

\begin{definition}
We define $\mathscr A_C^{\Diamond}$ to be the open subset of $\mathscr A_C^\heartsuit$
consisting of those points $a$ such that 
the corresponding spectral curve $C_a$ is 
smooth and 
irreducible. 
\end{definition}

\begin{corollary}\label{normality}
Assume the fibration $\pi:X\to C$ has only reduced fibers. Then 
the surface $\tilde Y_a$ is normal for $a\in\mathscr A_C^\Diamond$.
\end{corollary}
\begin{proof}
By the lemma above,
it suffices to show that $\tilde X_a$ is normal.
Since $\tilde X_a$ is Cohen-Macaulay, by Serre's criterion for normality,
it suffices to show that the $\tilde X_a$ is smooth in codimension $\leq 1$. 
The assumption implies the complement $X-X^0$ has codimension at least 2.
Since $C_a$ is smooth for $a\in\mathscr A_C^\Diamond$,
 the open subset $\tilde X^0_a:=\tC_a\times_CX^0\subset\tilde X_a$ is 
smooth (since the map $\tilde X^0_a\to C_a$ and $C_a$ are smooth) and 
the complement $\tilde X_a-\tilde X^0_a$ has codimension at least 2. The corollary follows
\end{proof}

\section{Ruled surfaces}\label{Ruled Surfaces}
Let $X$ be a ruled surface, that is, $X$ is a $\mathbb P^1$-bundle over 
a smooth projective curve $C$. We write $\pi:X\to C$ for the natural projection.

\begin{proposition}\label{spectral data for ruled}
We have $\mathscr A_C=\mathscr A_X$ and 
$\mathscr A_C^\heartsuit=\mathscr A_X^\heartsuit$.
The spectral data map 
\[\on{sd}_X:\on{Higgs}_X\to\mathscr A_X\] is surjective.

\end{proposition}
\begin{proof}
We first show that $\mathscr A_C=\mathscr A_X$.
Since $X$ is birational to $C\times\mathbb P^1$,
the embedding $j_\pi:\mathscr B_C=\bigoplus_{i=1}^n\rH^0(C,\Sym^i\Omega^1_C)\hookrightarrow
\bigoplus_{i=1}^n\rH^0(X,\Sym^i\Omega^1_X)=
\mathscr B_X$ 
in (\ref{maps between bases})
is an isomorphism. Since the maps in the diagram (\ref{maps between bases}) are all embeddings the desired equality follows. 
Since $\mathscr A_C^\heartsuit=\mathscr A_C\cap\mathscr A_X^\heartsuit$,
we obtain $\mathscr A_C^\heartsuit=\mathscr A_X^\heartsuit$.
We show that the spectral data map is surjective.
Let $a\in\mathscr A_C$. Then the push forward 
$E=\pi_{a*}\mO_{\tilde X_a}$ (here $\pi_{a}:\tilde X_a\to X$)
is a rank $n$ bundle on $X$. As the map $\pi_a$ factors through 
the morphism $i_a:\tilde X_a\to\rT_X^*$, the vector bundle $E$ is naturally 
equipped with a Higgs field $\theta$ and it follows from the construction that 
the spectral datum of the Higgs bundle $(E,\theta)$ is equal to $a$.
\end{proof}

We have the following:
\begin{proposition}
For every $a\in\mathscr A_C^\heartsuit$, the Cohen-Macaulay spectral surface 
$\tilde Y_a$ is 
isomorphic to $\tilde X_a=C_a\times_CX$, which is a closed subscheme of $\rT^*_X$,
and the Hitchin fiber 
$\Higgs_{X,a}$ is isomorphic to the moduli stack of Cohen-Macaulay 
sheaves on $\tilde Y_a$ of generic rank one. If $a\in\mathscr A_C^\Diamond$, then 
$\tilde Y_a$ is smooth and $\on{Higgs}_{X,a}$ is isomorphic to the Picard 
stack $\mathscr P_a$ of line bundles on $\tilde Y_a$.

\end{proposition}
\begin{proof}
The first claim follows from proposition \ref{Hitchin fibers} 
and lemma \ref{Y_a}.
The last claim follows from the fact that 
for $a\in\mathscr A_C^\Diamond$ the surface $\tilde Y_a=C_a\times_CX$ is smooth 
as 
$C_a$ and the fibration $\pi:X\to C$ are smooth.
\end{proof}

\begin{example}
Consider the case when $X=C\times\mathbb P^1$
and $n=2$. We have $\sA_X=\sA_C=\rH^0(C,\Omega_C^1)\oplus\rH^0(C,\Sym^2\Omega_C^1)$.
Let $a=(a_1,a_2)\in\sA^\heartsuit_C$ and $p_{a}:Y_a\to X$ be the corresponding
spectral surface. Then \'etale locally over $X$, 
the surface $Y_a$ is isomorphic to the closed subscheme of 
 $\mathbb A^4=\on{Spec}(k[x_1,x_2,t_1,t_2])$ 
defined by the equations 
\begin{equation}\label{phe2}
\left\{\begin{split} t_1^2+a_1t_1+a_2=0\\ t_2(2t_1+a_1)=0
 \\\ t_2^2=0
\end{split}
\right.
\end{equation}
here $x_1,x_2$ are local coordinate of $C$ and $\mathbb P_1$
and $a_i\in k[x_1]$.
Let $Dis=(a_1^2-4a_2=0)\subset C$ be the discriminant divisor for $a$.  
From (\ref{phe2}) we see that  $Y_a$ is an \'etale 
cover of degree $2$ away from the divisor  
$Dis\times\mathbb P^1\subset X$. Note  
that the spectral surface $p_{a}:Y_a\to X$ is not flat over $X$, indeed,
$p_{a*}\mO_{Y_a}$ has length three over 
$Dis\times\mathbb P^1$. 
The finite Cohen-Macaulayfication $\tilde Y_a\to Y_a$ is 
given by the flat quotient 
$\on{Spec}(p_{a*}\mO_{Y_a}/(p_{a*}\mO_{Y_a})^{\on{tors}})$
and is 
isomorphic 
to $\tilde Y_a=C_a\times\mathbb P^1$. 
The Hitchin fiber $\on{Higgs}_{X,a}$ is isomorphic to 
$\on{Higgs}_{C,a}\times\sP ic(\mathbb P^1)$, here 
$\on{Higgs}_{C,a}$ is the fiber of the Hitchin map for $C$ over $a$.

\end{example}

\section{Elliptic surfaces}\label{Elliptic surfaces}
In this section we study the case when $X$ is
an elliptic surface over an algebraically closed field $k$ of 
zero characteristic.
Namely, we assume there is a proper flat map 
$\pi:X\to C$ from $X$ to a smooth projective curve $C$ with general fiber 
a smooth curve of genus one. 
We will focus on the case when $\pi:X\to C$  
is non-isotrivial, relatively minimal, and has reduced fibers (e.g., semi-stable non-isotrivial elliptic surfaces).

\begin{proposition}\label{openness}
We have $\mathscr A_C^\heartsuit=\mathscr A_X^\heartsuit$
and 
$\mathscr A_C=\overline{\mathscr A_X^\heartsuit}$
is the closure of $\mathscr A_X^\heartsuit$ in $\mathscr A_X$. 
The image of the spectral data map 
$\on{sd}_X:\on{Higgs}_X\to\mathscr A_X$ 
is equal to $\mathscr A_C$. (The equalities here are understood set-theoretically).\footnote{If $\mathscr A_X$ is reduced then one can show 
that $\mathscr A_C^\heartsuit=\mathscr A_X^\heartsuit$,
$\mathscr A_C=\overline{\mathscr A_X^\heartsuit}$, and $\on{Im}(\on{sd}_X)=\mathscr A_C$
 scheme theoretically.}
\end{proposition}

The following proposition follows from 
proposition \ref{Hitchin fibers}, proposition \ref{Y_a}, and corollary \ref{normality}:

\begin{proposition}
For every $a\in\mathscr A_C^\heartsuit$, the Cohen-Macaulay spectral surface 
$\tilde Y_a$ is 
isomorphic to $\tilde X_a=C_a\times_CX$ and
the Hitchin fiber 
$\on{Higgs}_{X,a}$ is isomorphic to the moduli stack of Cohen-Macaulay 
sheaves on $\tilde Y_a$ of generic rank one.
If $a\in\mathscr A_C^\Diamond$, then 
$\tilde Y_a$ is normal and the Hitchin fiber 
$\on{Higgs}_{X,a}$ 
contains the Picard stack $\mathscr P_a$ of line bundles on
$\tilde Y_a$ whose rigidified neutral component $\mathscr P_a^0$ 
 is isomorphic to
a quotient 
of an abelian variety by $\mathbb G_m$ with trivial action.
 \end{proposition}

\begin{remark}
Unlike the case of ruled surfaces,  
the 
Cohen-Macaulay spectral surface $\tilde Y_a$ is in general not a 
sub-scheme of the cotangent bundle $\rT_X^*$ as the cotangent morphism 
$d\pi:\rT_C^*\times_CX\to\rT^*_X$ might not be an embedding.
\end{remark}
\begin{remark}
The results in section \ref{Ruled Surfaces} and section \ref{Elliptic surfaces}
show that Hitchin fibration for ruled and elliptic surfaces are 
closely related to Hitchin fibration for the base curve. This is compatible with 
the fact that under the Simpson correspondence, stable Higgs bundles 
for a smooth projective surface $X$
corresponds to irreducible representations of the fundamental group 
$\pi_1(X)$, and in the case of ruled and non-isotropic elliptic surfaces 
with reduced fibers we have $\pi_1(X)\is\pi_1(C)$ where 
$C$ is the base curve (see, e.g., \cite[Section 7]{Friedman}).
\end{remark}

The rest of the section is denoted to prove proposition \ref{openness}. 
We begin with some auxiliary lemmas:
\begin{lemma}
Let $\eta$ be the generic point of $C$ and 
$X_\eta=X\times_C\eta$ which is
an elliptic curve over $\eta$.
Consider the exact sequence 
$0\to\cT_{X^0/C}\to\cT_{X^0}\to(\pi^0)^*\cT_C\to 0$ (recall $X^0$ is the smooth locus of the fibration $\pi:X\to C$ and $\pi^0=\pi|_{X^0}$). Then the restriction of the exact sequence above to 
$X_\eta$ is isomorphic to the unique non-trivial self extension 
$0\to\mO_{X_\eta}\to\cI\to\mO_{X_\eta}\to 0$ of 
$\mO_{X_\eta}$.
\end{lemma}
\begin{proof}
Write $\cI=\cT_{X^0}|_{X_\eta}$. 
It is enough to show that 
$\on{dim}\rH^0(X_\eta,\cI)=1$.
Note that we have the following exact sequence 
\[0\to(\pi^0)_*\cT_{X^0/C}\to(\pi^0)_*\cT_{X^0}\to\cT_C\stackrel{\kappa}\to R^1\pi_*^0(\cT_{X^0/C}),\]
where the generic fiber of $\kappa$ is the Kodaira-Spencer morphism.
Since $\pi:X\to C$ is non-isotrivial, the map 
$\kappa$ is injective and the above exact sequence implies 
$(\pi^0)_*\cT_{X^0/C}\is (\pi^0)_*\cT_{X^0}$.
Since $X_\eta$ is an elliptic curve, 
we get 
\[\rH^0(X_\eta,\cI)=((\pi^0)_*\cT_{X^0})_\eta
\is((\pi^0)_*\cT_{X^0/C})_\eta\is
\rH^0(X_\eta,\cT_{X_\eta})\is\rH^0(X_\eta,\mO_{X_\eta})
\is k.\]
The lemma follows.

\end{proof}

\begin{lemma}\label{support 1}
Let 
$(E,\theta)\in\Higgs_X$. The action of 
$\cT_{X^0/C}\subset\cT_{X^0}$ on $E^0=E|_{X^0}$ is nilpotent.
Equivalently, let 
$\mM$ be coherent sheaf on $\rT^*_{X}$ associated to $(E,\theta)$.
Then the restriction of $\mM$ to $\rT^*_{X^0}$ is set theoretically supported on the closed subscheme
$\rT^*_{C}\times_{C}X^0\subset\rT^*_{X^0}$.
\end{lemma}
\begin{proof}
We need to show that the action of 
$\cT_{X^0/C}\subset\cT_{X^0}$ on $E^0=E|_{X^0}$ is nilpotent. For this, 
it suffices to show that, over the generic point $\eta$ of $C$, 
the action of $\mO_{X_\eta}\is (\cT_{X^0/C})_{X_\eta}\subset \cI=\cT_{X}|_{X_\eta}$ 
on $E_\eta=
E|_{X_\eta}$
is nilpotent.
Let $\phi\in\on{End}_{X_\eta}(E_\eta)$ be the 
image of $1\in\mO_{X_\eta}$ under the action map $\mO_{X_\eta}\is
(\cT_{X^0/C})_{X_\eta}\to\on{End}_{X_\eta}(E_\eta)$ and 
let $E_\eta=\bigoplus_{s\in\bar k_\eta}E_{\eta}(s)$
be the generalized eigenspaces decomposition for $\phi$ (here $k_\eta$ is the function field of $C$).
The action of $\cI$ on 
$E_\eta$ preserves each $E_{\eta}(s)$ and 
we denote by $\phi_s:\cI\to\on{End}_{X_\eta}(E_\eta(s))$ 
the action map. Consider   
the following composition 
\[w:\mO_{X_\eta}\to\cI\stackrel{\phi_s}\to\on{End}_{X_\eta}(E_\eta(s))\stackrel{\on{tr}}\to\mO_{X_\eta},\]
where 
$\on{tr}$ is the trace map. 
It follows from the definition of $E_\eta(s)$ that 
$w=s\cdot\on{Id}$.
On the other hand,
since $\cI$ is the unique nontrivial self extension of
$\mO_{X_\eta}$, we have $\mO_{X_\eta}\subset\on{ker}(f)$
for any map $f:\cI\to\mO_{X_\eta}$. It implies
$w=0$, hence
the eigenvalue $s$ must be zero. The lemma follows.
\end{proof}

\subsubsection*{Proof of proposition \ref{openness}}
We first show that the image of spectral data map $\on{sd}_X$ is equal to 
$\mathscr A_C$.
The same argument for the proof of proposition \ref{spectral data for ruled}
shows that $\mathscr A_C\subset\on{Im}(\on{sd}_X)$.
To show that $\on{Im}(\on{sd}_X)=\mathscr A_C$
it remains to check that
$a=\on{sd}_X(E,\theta)\in\mathscr A_C(F)$
for any geometric point
$(E,\theta)\in\on{Higgs}_X(F)$ (here $F$ is an algebraically closed field containing $k$).
Let $h_X(E,\theta)=(a_1,...,a_n)\in\mathscr B_X(F)$
be the image of $(E,\theta)$ under the Hitchin map
$h_X:\on{Higgs}_X\to\mathscr B_X=\bigoplus_{i=1}^n\rH^0(X,\Sym^i\Omega^1_{X})$. 
By the 
factorization of Hitchin map in (\ref{factorization-Hitchin}) and the
diagram in (\ref{maps between bases}), to show that 
$a\in\mathscr A_C(F)$
it suffices to check that 
\[\iota_X(a)=h_X(E,\theta)\in\mathscr B_C(F)=\bigoplus_{i=1}^n \rH^0(C_F,\Sym^i\Omega^1_{C_F}).\]
For this we observe that lemma \ref{support 1} implies the restriction $a_i|_{X_F^0}$ of each section $a_i$ to $X_F^0$ lies in the subspace 
\[a_i|_{X_F^0}\in\rH^0(X_F^0,\pi^*\Sym^i\Omega^1_{C_F}).\]
Since we assume $\pi$ has reduced fibers, the complement 
$X_F-X_F^0$ is a closed subset of codimension $\geq 2$ and it implies 
 \[a_i\in \rH^0(X_F,\pi^*\Sym^i\Omega^1_{C_F})=\rH^0(C_F,\Sym^i\Omega^1_{C_F}).\]
Hence $h_X(E,\theta)=(a_1,...,a_n)\in\mathscr B_C(F)$ and the desired claim follows.
We show that $\mathscr A_C^\heartsuit=\mathscr A_X^\heartsuit$. 
Since $\mathscr A_C^\heartsuit=\mathscr A_C\cap\mathscr A_X^\heartsuit$, we need to show that 
$\mathscr A_X^\heartsuit\subset\mathscr A_C=\on{Im}(\on{sd}_X)$.
This follows form 
the non-emptiness of the 
fiber $\on{Higgs}_{X,a}=\on{sd}_X^{-1}(a),\ a\in\mathscr A_X^\heartsuit$ established in proposition 4.3.
Finally, since $\mathscr A_C$ is irreducible and closed in 
$\mathscr A_X$ we have 
$\overline{\mathscr A_X^\heartsuit}=\overline{\mathscr A_C^\heartsuit}=\mathscr A_C$.
This completes the proof of proposition \ref{openness}.

\section*{Acknowledgement}

\foreignlanguage{vietnamese}{Ngô Bảo Châu}'s research is partially supported by NSF grant DMS-1702380 and the Simons foundation. He is grateful \foreignlanguage{vietnamese}{Phùng Hồ Hải} for stimulating discussions in an earlier stage of this project. He also thanks Gérard Laumon for many conversations on the Hitchin fibrations over the years and his encouragement. 
The research of
Tsao-Hsien Chen is partially supported by NSF grant DMS-1702337.


\begin{thebibliography}{999}

\bibitem{AK} A. Altman, S. Kleiman. {\it Compactifying the Picard scheme}. Advances in Mathematics {\bf 35}(1989), 50-112.


\bibitem{Arinkin} D. Arinkin. {\it Autoduality of compactified Jacobians for curves with plane singularities}. Journal of algebraic geometry {\bf 22}(2013), 363-388.

\bibitem{Beauville} A. Beauville. {\it Surfaces algébriques complexes.} Astérisque, No. {\bf 54}. Société Mathématique de France, Paris 1978. iii+172 pp.

\bibitem{BNR} A. Beauville, M.S. Narasimhan, and S. Ramanan. {\it Spectral curves and the generalized theta divisor.} Journal für die Reine und Angewante Mathematik (1989).

\bibitem{BLR} S. Bosch, W. Lutkebohmert, M. Raynaud. {\it N\'eron Models.} Ergebnisse der Math., 3. Folge, Bd. 21, 1990.



\bibitem{BBG} J. Bernstein, A. Braverman, D. Gaitsgory. {\it The Cohen-Macaulay property of the category of $(\fg,K)$-modules}. Sel. math., New ser. {\bf 3} (1997) 303 – 314.

\bibitem{BH} W. Bruns, Herzog. {\it Cohen-Macaulay rings}.  Cambridge Studies in Advanced Mathematics {\bf 39}, Cambridge University Press 1993. 

\bibitem{Friedman} R. Friedman. {\it Algebraic surfaces and holomorphic vector bundles}. Universitext. Springer, Berlin, 1998.

\bibitem{Geisser} T. Geisser. {\it The affine part of the Picard scheme}. Compositio Mathematica {\bf 145}(2009) 415-422.

\bibitem{GG} W.L Gan, V. Ginzburg. {\it Almost commuting variety, D-modules, and 
Cherednik algebras}. International Mathematics Research
Papers, 2, (2006), p. 1-54.


\bibitem{Hitchin} N.J. Hitchin. {\it Stable bundles and integrable systems.} Duke Mathematical Journal {\bf 54} (1987) 91-114. 

\bibitem{HR} M. Hochster, J.L. Roberts. {\it Rings of invariants of reductive groups acintg on regular rings are Cohen-Macaulay}. Advances in Mathematics {\bf 13}(1974) 117-175

\bibitem{Hunziker} M. Hunziker. {\it Classical invariant theory for finite reflection groups}. Transformation groups {\bf 2}(1997) 147-163.

\bibitem{Joseph} A. Joseph. {\it On a Harish-Chadra homomorphism}. Comptes rendus de l'Académie des sciences de Paris. {\bf 324}(1997) 759-764.

\bibitem{Ngo} B.C. Ngô. {\it Fibrations de Hitchin et endoscopie}. Inventiones mathematicae. {\bf 164}(2006) 399-453.

\bibitem{Serre} J.-P. Serre. {\it Prolongement de faisceaux analytiques cohérents}. Annales de l'Institut Fourier. {\bf 16}(1966), 363-374.

\bibitem{Simpson 1} C. Simpson. {\it Higgs bundles and local systems}. Publications mathématiques de l'IHES. {\bf 75}(1992), 5-95.

\bibitem{Simpson 2} C. Simpson. {\it Moduli of representations of the fundamental groups of a smooth projective variety II}. Publications mathématiques de l'IHES. {\bf 80}(1994), 5-79.

\end{thebibliography}
\end{document}